\documentclass{amsart}
\usepackage[utf8]{inputenc}
\usepackage[colorlinks=true,citecolor=black,linkcolor=black,urlcolor=blue]{hyperref}
\usepackage{amsthm,amsmath,amssymb}
\usepackage{xspace}
\usepackage{pgf,tikz}
\usetikzlibrary{arrows}

\def\Chip{{\rm Chip}}

\def\per{{\rm per}}

\newcommand{\cl}[1]{\mbox{\ensuremath{\mathbf{#1}}}\xspace}
\def\NP{\cl{NP}}
\def\coNP{\cl{co}-\cl{NP}}
\newenvironment{bekezdes}[1]{\vspace{0.5cm} {\textbf{#1} }}{\vspace{0.5cm}}

\theoremstyle{plain}
\newtheorem{thm}{Theorem}

\newtheorem{lemma}[thm]{Lemma}
\newtheorem{prop}[thm]{Proposition}

\newtheorem*{claim*}{Claim}

\theoremstyle{definition}
\newtheorem{defn}[thm]{Definition}
\newtheorem{exmp}[thm]{Example}

\newtheorem{prob}[thm]{Problem}

\theoremstyle{remark}
\newtheorem{remark}[thm]{Remark}

\title{On the complexity of the chip-firing reachability problem}

\author{B\'alint Hujter}
\address{MTA-ELTE  Egerv\'ary  Research  Group, 
Department  of  Operations  Research, 
E\"otv\"os  Lor\'and University,
Budapest, Hungary.}
\email{hujterb@cs.elte.hu}
\thanks{Supported  by  the  Hungarian  Scientific  Research  Fund  -  OTKA 
K109240.}
\author{Viktor Kiss}
\address{Department  of  Analysis,
E\"otv\"os  Lor\'and University,
Budapest, Hungary.}
\email{kivi@cs.elte.hu}
\thanks{Supported  by  the  Hungarian  Scientific  Research  Fund  -  OTKA 
104178, 113047.}
\author{Lilla T\'othm\'er\'esz}
\address{MTA-ELTE  Egerv\'ary  Research  Group,
Department  of Computer Science,
E\"otv\"os  Lor\'and University,
Budapest, Hungary.}
\email{tmlilla@cs.elte.hu}
\thanks{Supported  by  the  Hungarian  Scientific Research  Fund  -  OTKA 
K109240.}

\begin{document}

\date{}

\begin{abstract}
In this paper, we study the complexity of the chip-firing reachability problem.
We show that for Eulerian digraphs, the reachability problem can be decided in strongly polynomial time, even if the digraph has multiple edges. 
We also show a special 
case when the reachability problem can be decided in polynomial time for general digraphs: if the target distribution is recurrent restricted to each strongly connected component. 
As a further positive result, we show that the chip-firing reachability problem is in \coNP for general digraphs. 
We also show that the chip-firing halting problem is in \coNP for Eulerian digraphs.
\end{abstract}

\maketitle

\renewcommand{\thefootnote}{}



\renewcommand{\thefootnote}{\arabic{footnote}}
\setcounter{footnote}{0}

\section{Introduction}
\label{s:intro}

Chip-firing is a solitary game on a directed graph, defined by Bj\"orner, Lov\'asz and Shor \cite{BLS91}.  
Each vertex contains a pile of chips. A legal move is to choose a vertex with at least as many chips as its out-degree and let it send  a chip along each outgoing edge.
We analyze the complexity of the following reachability question: given two chip-distributions $x$ and $y$, decide whether $y$ can be reached from $x$ by playing a legal game. 
This question is a special case of the reachability problem for integral vector addition systems \cite{BL92}. It was first considered by Bj\"orner and Lov\'asz, who gave an algorithm that decides the reachability question and runs in weakly polynomial time for simple Eulerian digraphs \cite{BL92}.
The complexity of the reachability problem was left open for Eulerian digraphs with multiple edges, and for non-Eulerian digraphs. The question whether the reachability problem is in \NP or in \coNP was also left open. (Informally: whether there exists a short certificate for the reachability or for the non-reachability.)

In this paper, we show that the chip-firing reachability problem can be decided in polynomial time for Eulerian digraphs with multiple edges. 
Our algorithm is strongly polynomial.
The main ingredient of the algorithm is a lemma ensuring that if one chip-distribution is reachable from another, then it can be reached by firing an ``ascending chain of sets of vertices''.

For general digraphs, we show that the chip-firing reachability problem is in \coNP.
Also, we show a special case when the chip-firing reachability problem is polynomial time solvable even on non-Eulerian digraphs. 
If $G$ is a strongly connected digraph and the chip-distribution $y$ is recurrent, i.e., it is reachable from itself by a non-empty legal game, then we characterize the set of chip-distributions from which $y$ is reachable. 
This characterization enables one to decide in polynomial time whether a recurrent chip-distribution $y$ is reachable by a legal game from a given chip-distribution $x$.
For weakly connected directed graphs, we generalize the characterization theorem for chip-distributions $y$ that are recurrent restricted to each strongly connected component. This theorem also gives rise to a polynomial algorithm.

Finally, in Section \ref{s:open_prob}, we collect some open problems related to the reachability problem.
In this last section, we show that the chip-firing halting problem (which is a problem similar to the chip-firing reachability problem) is in $\NP\cap\coNP$ for Eulerian digraphs, which makes it a good candidate for the search of a polynomial algorithm.

\section{Preliminaries}
\label{s:prelim}

\subsection{Digraphs}

Throughout this paper, \emph{digraph} means a (weakly) connected directed graph that can have multiple edges but no loops.  
A digraph is usually denoted by $G$. The vertex set and edge set of a digraph $G$ are denoted by $V(G)$ and $E(G)$ (or simply $V$ and $E$), respectively.
For a vertex $v$, the in-degree and the out-degree of $v$ are
denoted by $d^-(v)$ and $d^+(v)$, respectively.
We denote a directed edge leading from vertex $u$ to vertex $v$ by $\overrightarrow{uv}$.
The multiplicity of a directed edge $\overrightarrow{uv}$ is denoted by $\overrightarrow{d}(u, v)$.

A digraph is \emph{simple}, if $\overrightarrow{d}(u, v)\leq 1$ and $\overrightarrow{d}(v, u)\leq 1$ for each pair of vertices $u,v\in V$. 
A digraph is \emph{Eulerian}, if $d^+(v)=d^-(v)$ for each $v\in V$.
A digraph is \emph{strongly connected}, if for each pair of vertices $u,v$, there is a directed path from $u$ to $v$, and also from $v$ to $u$. 
A connected Eulerian digraph is always strongly connected.
Each digraph has a unique decomposition into strongly connected components. 
A component is called a \emph{sink component}, if there is no edge leaving the component. Note that a digraph always has at least one sink-component.

Throughout this paper, we identify undirected graphs with the digraph obtained by replacing each edge with a pair of oppositely directed edges. This way, undirected graphs become special Eulerian digraphs.

If we give a digraph as an input to an algorithm, we always encode it by its adjacency matrix. Hence the size of the input is not increased by the values of the edge multiplicities, just the logarithms of them. An algorithm runs in \emph{polynomial time} if the number of basic steps it makes is bounded by a polynomial in the size of the input. If the input of an algorithm consists of integer numbers, we can talk about strongly polynomial running time. An algorithm runs in \emph{strongly polynomial time} if the following two conditions are satisfied: 
\begin{itemize}
  \item[1.] 
  in the model, where basic arithmetic operations (addition, subtraction, multiplication, division, and comparison) take a unit time step to perform,
its running time is bounded by a polynomial in the number of integers contained in the input;
  \item[2.] the space used by the algorithm is bounded by a polynomial in the size of the input.
\end{itemize}
For a more detailed explanation, see \cite[Chapter 1.3]{groetschel-lovasz-shrijver}.

We denote by $\mathbb{Z}^{V}$ the set of integer vectors indexed by the vertices of a digraph $G$. 
$\mathbb{Z}_+^V$ denotes the set of vectors with nonnegative integer coordinates. 
For $S\subseteq V$, we denote the characteristic vector of $S$ by $\mathbf{1}_S$.
If $S=\{v\}$, we use the notation $\mathbf{1}_v$. 
We denote the vector with each coordinate equal to zero by $\mathbf{0}$.

The \emph{Laplacian} of a digraph $G$ is the
following matrix $L \in \mathbb{Z}^{V \times V}$:
\[
L(u,v) = \left\{\begin{array}{cl} -d^+(v) & \text{if } u=v, \\
\overrightarrow{d}(v, u) & \text{if } u\neq v.      
\end{array} \right.
\]

\subsection{Chip-firing}

In a chip-firing game we consider a digraph $G$ with a pile of chips on each of its nodes. A position of the game, called a \emph{chip-distribution} (or just distribution) is described by a vector $x \in \mathbb{Z}_+^V$, where $x(v)$ is interpreted as the number of chips on vertex $v \in V$. We denote the set of all chip-distributions on $G$ by $\Chip(G)$.

The basic move of the game is \emph{firing} a vertex. It means that this vertex passes a chip to its neighbors along each outgoing edge, and so its number of chips decreases by its out-degree. In other words, firing a vertex $v$ means taking the new chip-distribution $x + L\mathbf{1}_v$ instead of $x$.

The firing of a vertex $v \in V$ is \emph{legal} with respect to a distribution $x$, if $v$ has a nonnegative amount of chips after the firing (i.e.\ $x(v) \geq d^+(v)$). A \emph{legal game} is a sequence of distributions in which every distribution is obtained from the previous one by a legal firing. For a legal game, let us call the vector $f \in \mathbb{Z}_+^V$, where $f(v)$ equals the number of times $v$ has been fired, the \emph{firing vector} of the game.
A game terminates if no firing is legal with respect to the last distribution.
The following theorem of Bj\"orner, Lov\'asz and Shor describes a very important ``Abelian'' property of the chip-firing game.

\begin{thm}\cite[Remark 2.4]{BLS91} \label{t:chip-firing_commutative}
From a given initial chip-distribution, either every legal game can be continued indefinitely, or every legal game terminates after finitely many steps. The firing vector of every maximal legal game is the same.
\end{thm}
Based on this fact, we call a distribution $x$ \emph{terminating} if a legal game (hence, all legal games) started from $x$ terminates, and we call $x$ \emph{non-terminating} otherwise. 

For a given vector $b \in \mathbb{Z}_+^{V}$, let us call the following 
game \emph{chip-firing game with upper bound $b$}: We are only allowed to 
make legal firings, and each vertex $v$ can be fired at most $b(v)$ times 
during the whole game. Bj\"orner and Lov\'asz show the ``Abelian'' property for the bounded chip-firing game as well. 

\begin{lemma} \cite[Lemma 1.4]{BL92} \label{l:korlatos_jatek_moho}
For a given bound $b \in \mathbb{Z}_+^V$ and initial distribution $x$, each maximal 
bounded game with upper bound $b$ and initial distribution $x$ has the same firing vector.
\end{lemma}

A nonnegative vector $p \in \mathbb{Z}_+^V$ is called a \emph{period vector} for $G$ if  $Lp=\mathbf{0}$. A non-zero period vector is called \emph{primitive} if its entries have no non-trivial common divisor. The following proposition follows from \cite[3.1 and 4.1]{BL92}.

\begin{prop} \label{prop::period}
For a strongly connected digraph $G$ there exists a unique primitive period vector $p_G$, moreover, it is strictly positive. If $G$ is connected Eulerian, then $p_G = \mathbf{1}_V$.
For a general digraph $G$, if $G_1, \dots, G_k$ are the sink components of $G$ and a vector $z \in \mathbb{Z}^V$ satisfies $Lz = \mathbf{0}$ then $z = \sum_{i=1}^k \lambda_i p_i$, where for $i \in \{1, \dots, k\}$, $\lambda_i\in \mathbb{Z}$ and $p_i$ is the primitive period vector of $G_i$ restricted to $V(G_i)$ and zero 
elsewhere.
\end{prop}

For a strongly connected digraph $G$, let us denote the sum of the coordinates of $p_G$ by $\per(G)$. For a general digraph $G$ let $\per(G)=\sum_{i=1}^\ell \per(G_i)$ where $G_1,\dots, G_\ell$ are the strongly connected components of $G$.  

\subsection{The reachability problem}

A basic question about the chip-firing game is the so-called reachability question: Given two 
chip-dis\-tri\-bu\-ti\-ons $x,y\in \Chip(G)$, is it possible to reach $y$ from 
$x$ by playing a legal game? Let us denote by $x \leadsto y$ if such a legal 
game exists.

Our main goal in this paper is to investigate the computational complexity of the reachability question.
Let us sum up the previous results about the problem. To do so, we state an important lemma of Bj\"orner and Lov\'asz.

 \begin{lemma} \cite[Lemma 4.3]{BL92} \label{l:per_vektor_kihagyhato}
   Let $p$ be a period vector of a digraph $G$, and suppose that $\alpha=(v_1, v_2, \dots, v_s)$ is a legal sequence of firings on $G$ from some initial distribution. Let $\alpha'$ be the sequence obtained from $\alpha$ by deleting the first $p(v)$ occurrences of each vertex $v$ (if $v$ occurs less than $p(v)$ times in $\alpha$, then we delete all of its occurrences). Then $\alpha'$ is also a legal sequence of firings from the same initial distribution.
 \end{lemma}

A vector $f\in \mathbb{Z}_+^V$ is called \emph{reduced} if $f \not \ge p$ for every non-zero period vector $p$. The following phenomenon is a direct consequence of the previous lemma:

\begin{lemma}\cite[Lemma 5.2]{BL92} \label{l:reach_with_reduced}
If $x\leadsto y$, then there exists a legal game transforming $x$ to $y$ with a reduced firing vector.
\end{lemma}

Note that if $x \leadsto y$ then for the firing vector $f$ of a legal game transforming $x$ to $y$, $y = x + Lf$. Among the vectors $g \in \mathbb{Z}_+^V$ satisfying $y = x + Lg$, there is a unique one that is reduced.

The previous lemmas imply, that the reachability question can be decided ``greedily'': For given $x,y\in \Chip(G)$ one can decide if there exists a reduced vector $f$ with $y = x + Lf$. If no such vector exists then $x \not \leadsto y$. If such a vector $f$ exists, it can be computed. By Lemma \ref{l:reach_with_reduced}, $x\leadsto y$ if and only if there is a legal game from $x$ to $y$ with firing vector $f$. By Lemma \ref{l:korlatos_jatek_moho}, we can find greedily a maximal chip-firing game from $x$ with upper bound $f$. We have $x\leadsto y$ if and only if this maximal bounded chip-firing game has firing vector $f$.

This reasoning gives an algorithm for deciding the reachability problem, however, this algorithm is in general not polynomial, as the firing vector $f$ may have exponentially large elements.

Bj\"orner and Lov\'asz improve this naive algorithm, and obtain the following:

\begin{thm}[\cite{BL92}]
There is an algorithm that for given $x,y \in \Chip(G)$ on a digraph $G$ decides whether $x \leadsto y$ holds, and runs in
\[ O(|V|^2D^2\per(G)\log(|V| D N \per(G))) \] time, where $D=\max\{d^+(v): v\in V\}$ and $N$ is the number of chips in $x$.
\end{thm}

This algorithm is not polynomial in general, as $\per(G)$ and $D$ may be exponentially large. However, as for simple Eulerian digraphs, $\per(G)=|V|$ and $D \le |V|$, the algorithm is weakly polynomial for simple Eulerian digraphs.

In this paper, we show that the reachability problem can be decided in polynomial time also for general Eulerian digraphs (i.e.\ also for Eulerian digraphs with multiple edges). In addition, our algorithm is strongly polynomial. 
For general digraphs, we show that the reachability problem is in \coNP.
We also show that in the special case that $y$ is recurrent restricted to each strongly connected component, whether $x\leadsto y$ holds can be decided in polynomial time for general digraphs.

\subsection{Two necessary conditions}

Before presenting our results, let us describe two simple, polynomial-time computable necessary conditions for the reachability problem:

The first necessary condition is the linear equivalence of $x$ and $y$:
\begin{defn}[Linear equivalence \cite{BN-Riem-Roch}]
For $x, y \in \mathbb{Z}^V$, let $x\sim y$ if there exists $z\in \mathbb{Z}^V$ such that $x = y + Lz$.  In this case we say that $x$ and $y$ are linearly equivalent.
\end{defn}
One can easily check that $\sim$ defines an equivalence relation on $\mathbb{Z}^V$. This equivalence relation has been defined by Baker and Norine \cite{BN-Riem-Roch}, and it plays an important role in the theory of graph divisors.
It is easy to see, that $x\sim y$ is a necessary condition for $x\leadsto y$. Indeed, if there is a legal game leading from $x$ to $y$, let its firing vector be $f$. Then $y=x+Lf$.

A slightly stronger necessary condition for $x\leadsto y$ is the existence of $f\in \mathbb{Z}_+^V$ such that $y=x+Lf$. (This condition is stronger than the linear equivalence only in that we require $f$ to be nonnegative.) 
As the firing vector of a legal game from $x$ to $y$ gives a nonnegative $f$ such that $y=x+Lf$, this condition is also necessary. 
For strongly connected digraphs, the two conditions are equivalent, since a strong digraph has a period vector that is positive on every coordinate. 

By \cite[Theorem 1.4.21]{groetschel-lovasz-shrijver}, whether $x\sim y$ can be decided in polynomial time, and if the answer is yes, a vector $f\in  \mathbb{Z}^V$ such that $y=x+Lf$ can be computed. We show that the second necessary condition can also be decided in polynomial time, and in the case of Eulerian digraphs, even in strongly polynomial time.
\begin{prop}\label{prop::linequi.strongly.poly}
There is a polynomial algorithm that for a given digraph $G$ and $x,y\in Chip(G)$ decides whether there exists a nonnegative integer vector $f$ such that $y = x + Lf$, and if such a vector exists, it computes one. \\
In the case of Eulerian digraphs, this can be done in strongly polynomial time.
\end{prop}
\begin{proof}
First, let $G$ be Eulerian. If the number of chips in $x$ and $y$ are different, then there is no such $f$. Now suppose that $\sum_{v\in V}x(v)=\sum_{v\in V}y(v)$.
As a connected Eulerian digraph is strongly connected, the Laplacian matrix $L$ of $G$ has a one-dimensional kernel, and for an arbitrary vertex $v\in V(G)$, the matrix $L_v$ obtained from $L$ by deleting the row and column corresponding to $v$ is nonsingular. We can compute $L_v^{-1}$ in strongly polynomial time \cite[Corollary 1.4.9]{groetschel-lovasz-shrijver}. Let $x_v$ and $y_v$ be the vectors we get from $x$ and $y$ by deleting the coordinate corresponding to $v$, respectively.
Let $g\in\mathbb{R}^V$ be the vector with
$$g(u)=
\left\{\begin{array}{cl} 
(L_v^{-1}(y_v-x_v))(u) & \text{if } u\neq v \\
        0 & \text{if } u=v.
\end{array} \right.$$

It is easy to see that $y(u) = (x + Lg)(u)$ for each $u \neq v$, and since $\sum_{u\in V}x(u)=\sum_{u\in V}y(u)$, we have $y=x+Lg$. 
The coordinates of $g$ are not necessarily integer. All the vectors $f$ such that $y=x+Lf$ are of the form $g+c\cdot \per_G$, therefore we need to decide if there is a nonnegative integer vector of this form.
As now $G$ is Eulerian, $\per_G=\mathbf{1}_G$. Since $g(v)=0$, $c$ needs to be an integer, hence $g$ also needs to be an integer vector. If $g$ is an integer vector, $c=\min\{-g(u): u\in V\}$ is an appropriate choice.

If $G$ is not Eulerian, we proceed with the following polynomial, although not strongly polynomial algorithm. By \cite[Theorem 1.4.21]{groetschel-lovasz-shrijver}, we can decide in polynomial time if the equation $Lg=y-x$ has an integer solution, and if it does, compute one. By Proposition \ref{prop::period}, a nonnegative solution exists if and only if the $g$ we got from solving $Lg=y-x$ has nonnegative coordinates on the non-sink components. If $g$ is nonnegative on the non-sink components, we can make it nonnegative by adding an appropriate period vector to it.
\end{proof}

\section{A polynomial algorithm for Eulerian digraphs}
\label{s:Euler_multi}

In this section, we describe our algorithm deciding the reachability problem on Eulerian digraphs, which is polynomial even for graphs with multiple edges. The following lemma is the heart of our algorithm. Informally, it says, that if one chip-distribution is reachable from another, then it can be reached so that we fire an ascending chain of subsets of vertices.

\begin{lemma} \label{l:growing_sets}
Let $G$ be an Eulerian digraph. Suppose that we have two chip-distributions $x$ and $y$ such that $x\leadsto y$. Then    
there exists a sequence of legal firings 
$(v_1, v_2, \dots ,v_s)$ 
that transforms $x$ to $y$, and there exist indices 
$i_0=0,i_1,i_2, \dots i_t=s$ 
such that for each $j = 1, \dots, t$, no vertex appears twice in the sequence $v_{i_{j-1}+1}, \dots, v_{i_j}$, and by setting 
$S_j = \{v_{i_{j-1}+1}, \dots, v_{i_j}\}$, 
we have 
$S_1 \subseteq S_2 \subseteq \dots \subseteq S_t \subsetneq V$.
\end{lemma}

 \begin{proof}
   Lemma \ref{l:per_vektor_kihagyhato} plays a key role in this proof. Note that for Eulerian digraphs, Lemma \ref{l:per_vektor_kihagyhato} says, that if for a legal sequence of firings, we leave out the first occurrence of each vertex that occurs in the sequence, we still get a legal game.
   
   Since $x\leadsto y$, there exists a sequence of firings $\alpha=(w_1,\dots w_s)$ that is legal, and transforms $x$ to $y$. Let $f_{\alpha}\in\mathbb{Z}_+^V$ be the firing vector of $\alpha$, and $S_{\alpha}$ be the set of vertices that occur at least once in $\alpha$.
	Notice that we can suppose that $S_{\alpha}\neq V$. Indeed, if $S_{\alpha}=V$ that means that each vertex occurs in $\alpha$. Then by Lemma \ref{l:per_vektor_kihagyhato}, leaving out the first occurrence of each vertex, we still get a legal game $\alpha'$ with firing vector $f_{\alpha}-\mathbf{1}_V$. As $L f_{\alpha}= L(f_{\alpha}-\mathbf{1}_V)$, this game still transforms $x$ to $y$. We can continue this until there is a vertex that does not occur in our legal sequence of firings, hence we can suppose $S_{\alpha}\neq V$.

	We use induction for the largest coordinate of $f_{\alpha}$. If the largest coordinate of $f_{\alpha}$ is one, then we have $t=1$ and $S_1=S_\alpha \neq V$, and we are ready.   
   
	Suppose that the largest coordinate of $f_{\alpha}$ is larger than one. 
We prove that from initial distribution $x$, there exists a legal sequence of firings $\alpha'$ with firing vector $f_{\alpha'}=f_{\alpha}-\mathbf{1}_{S_\alpha}$ that can be extended legally by a sequence $\beta$ of firings with firing vector $\mathbf{1}_{S_\alpha}$. Indeed, 
by Lemma \ref{l:per_vektor_kihagyhato}, the sequence of firings $\alpha'$ that we get from $\alpha$ by deleting the first occurrence of each vertex is still legal. The firing vector of this sequence is $f_{\alpha}-\mathbf{1}_{S_{\alpha}}$. Play the bounded chip-firing game with upper bound $f_{\alpha}$ from initial distribution $x$. Then $\alpha'$ is a valid beginning. As $\alpha$ is a legal chip-firing game with upper bound $f_{\alpha}$, and its firing vector is $f_{\alpha}$, by Lemma \ref{l:korlatos_jatek_moho}, each maximal bounded chip-firing game with upper bound $f_{\alpha}$ has firing vector $f_{\alpha}$. Hence $\alpha'$ can be extended to a legal game with firing vector $f_{\alpha}$. 
Let us call the sequence of the last $|S_{\alpha}|$ firings $\beta$.
   
   The largest coordinate of $f_{\alpha'}$ is strictly smaller than the largest 
   coordinate of $f_{\alpha}$, hence by the induction hypothesis, there is a 
   legal sequence $\gamma=(v_1,\dots, v_{s'})$ of firings with firing vector 
   $f_{\alpha'}$ such that there exist indices $i_0=0,i_1,i_2, \dots i_t=s'$ 
   such that for each $j=1,\dots, t$, no vertex appears twice in the 
   sequence 
   $v_{i_{j-1}+1}, \dots v_{i_j}$, and after setting $S_j = \{v_{i_{j-1}+1}, 
   \dots v_{i_j}\}$, we have $S_1 \subseteq S_2 \subseteq \dots \subseteq S_t$.
As $\alpha'$ can be legally extended by $\beta$,
$\gamma$ can also be legally extended by $\beta$, since the chip-distribution after a sequence of firings only depends on the firing vector, which is the same for $\gamma$ and for $\alpha'$. Now let $i_{t+1}=s$, thus $S_{t+1}=S_{\alpha}$.
As $S_t\subseteq S_{\alpha'}\subseteq S_{\alpha}\neq V$, we proved the statement for $\alpha$.
\end{proof}

\begin{remark}
'Ascending chains' also play a role in the related field of graph divisor theory, see for example \cite[Lemma 1.3.]{Dion_treewidth} or the notion of 'level sets' in \cite{vDdB.bsc.thesis}. There is also an analogous lemma for the model of cluster firing, see Section 4 in \cite{Holroyd08}.
\end{remark}

\begin{thm}
There is a strongly polynomial algorithm that decides whether $x \leadsto y$ for two chip-distributions $x$ and $y$ on an Eulerian digraph $G$. 
\end{thm}
\begin{proof} 
  The idea of the proof is to use Lemma \ref{l:growing_sets}:
  $x\leadsto y$ if and only if $y$ can be reached from $x$ by firing an ``ascending chain of vertex sets''. Since the chain is ascending, the number of distinct sets is at most $|V|$ (but one may occur in the  sequence exponentially many times). By Proposition \ref{prop::linequi.strongly.poly}, we can decide whether there exists a reduced vector $f$ such that $y=x+Lf$, and if so, we can compute what these ever growing sets should be. The main idea is that it is enough to check for each set, whether it can be fired at its last occurrence.
  
Let us write this formally.
The algorithm is the following:

Using the procedure of Proposition \ref{prop::linequi.strongly.poly}, we decide whether there exists an integer vector $g \in \mathbb{Z}^V_+$ such that $y = x + L g$. If no such vector exists then $x \not \leadsto y$. If such a $g$ exists, then let $k = \min_{v \in V} g(v)$ and $f = g - k\cdot \mathbf{1}_V$. 
Since $L \mathbf{1}_V = \mathbf{0}$, $y = x + L f$. 
  Moreover, the coordinates of $f$ are nonnegative, and it has a coordinate 
  that is zero. Let $t = \max_{v \in V} f(v)$ and for $1 \le j \le t$, let 
  $S_j = \{v \in V : f(v) \ge t - j + 1\}$. It is easy to see 
  that $S_1 \subseteq S_2 \subseteq \dots \subseteq S_{t} \subsetneq V$.
  Let $k$ be the number of distinct $S_j$'s, and for $1 \le i \le k$ let $a_i$ be the index of the first occurrence of the $i$'th smallest set among the  $S_j$'s. Also, set $a_{k + 1} = t + 1$. With these notations, $S_{j} = S_{a_i}$ if $a_i \le j < a_{i + 1}$. 
  
  Now let $x_1 = x$ and define $x_{j} = x + \sum_{\ell = 1}^{j-1} L \mathbf{1}_{S_\ell}$
  for $j = 1, \dots ,t+1$. 
  We do not compute all of these chip-distributions (as there can be 
  exponentially many), but note that  
  for a fixed $j$, $x_j$ can be 
  computed in polynomial time: If $a_i \le j < a_{i + 1}$ then 
  \[
  x_j = x_1 + L \left((j - a_i) \mathbf{1}_{S_{a_i}} + 
  \sum_{\ell = 1}^{i - 1} (a_{\ell + 1} - a_\ell)\mathbf{1}_{S_{a_\ell}} \right).
  \]
  Now the algorithm proceeds as follows: For each   
  $1 \le i \le k$ compute $x_{a_{i + 1} - 1}$ and $x_{a_{i + 1}}$ and check 
  whether $x_{a_{i + 1} - 1} \leadsto x_{a_{i + 1}}$.
  This can also be done in polynomial time, since by Lemma \ref{l:korlatos_jatek_moho} and Lemma \ref{l:reach_with_reduced} for each $i$ we only need to check greedily whether the firing vector $\mathbf{1}_{S_{a_i}}$ can be fired from initial distribution $x_{a_{i + 1} - 1}$.
  If $x_{a_{i + 1} - 1} \leadsto x_{a_{i + 1}}$ for each $1 \le i \le k$, 
  then the algorithm returns $x \leadsto y$, otherwise the algorithm returns 
  $x \not \leadsto y$.   
  
  Now we prove the correctness of the algorithm.
  First we prove that if the algorithm returns $x \leadsto y$ then 
  $x \leadsto y$.
  
  Note that $f = \sum_{j = 1}^{t} \mathbf{1}_{S_j}$, hence $x_{t + 1} = 
  x + \sum_{j = 1}^{t} L \mathbf{1}_{S_j} = x + L f = y.$
  Thus for proving $x\leadsto y$, it is enough to prove for each $1\leq j\leq 
  t$ that $x_j \leadsto x_{j + 1}$. So let $1 \le j \le t$, then for some $i 
  \le k$, $a_i \le j < a_{i + 1}$. Hence $x_{j + 1} = x_j + L 
  \mathbf{1}_{S_{a_i}}$.  
  
  Since the algorithm returned $x \leadsto y$, we have $x_{a_{i + 1} - 1} \leadsto x_{a_{i + 1}}$. Let $\beta$ be a legal game witnessing 
  this. By Lemma 
  \ref{l:per_vektor_kihagyhato}, we can suppose that $\beta$ has firing vector 
  $\mathbf{1}_{S_{a_i}}$. We prove that $\beta$ is also a legal game starting 
  from the distribution $x_j$. For this, it is enough to show that $x_j(v) \ge 
  x_{a_{i + 1} - 1}(v)$ 
  for each $v \in S_j = S_{a_i}$. But this is true, since $G$ is Eulerian, and 
  $x_{a_{i+1} - 1} = x_j + L \cdot (a_{i + 1} - 1 - j)\mathbf{1}_{S_j}$. 
  Hence each 
  vertex $v \in S_j$ fires $(a_{i + 1} - 1 - j) \cdot d^+(v)$ chips and gains
  \[(a_{i + 1} - 1 - j) \cdot \sum_{u \in S_j} \overrightarrow{d}(u, v) \leq (a_{i + 1} - 1 - j) \cdot d^-(v) = (a_{i + 1} - 1 - j) \cdot d^+(v)
  \]
  chips as we change from $x_j$ to $x_{a_{i + 1} - 1}$. Thus $x_j \leadsto x_{j 
  + 1}$ for every $j \le t$, hence $x \leadsto y$. 

  Now it remains to show that if $x \leadsto y$ then the algorithm returns 
  $x \leadsto y$. 
  Let $f'$ be the firing vector of the legal game $(v_1, v_2, \dots ,v_s)$ from $x$ to $y$ provided by 
  Lemma \ref{l:growing_sets}. One can easily see that $f' = f$, hence the sets 
  $S_1, \dots, S_t$ coincide with those given by the lemma. 
  Consequently, after firing $(v_1,\dots v_{i_{j-1}})$ from initial distribution $x$, we arrive at $x_j$. If we continue the firing of the  legal sequence and fire $(v_{i_{j-1}+1}, \dots v_{i_{j}})$, then we arrive at $x_{j+1}$. Hence $(v_{i_{j-1}+1}, \dots v_{i_{j}})$ is a legal game from initial distribution $x_j$, showing $x_j\leadsto x_{j+1}$.
In particular, $x_{a_{i + 1} - 1} \leadsto x_{a_{i + 1}}$ for each $1 \leq i 
\le k$, hence the algorithm returns $x \leadsto y$. 

The algorithm runs in strongly polynomial time. 
\begin{enumerate}
	\item Preprocessing: computing the (diagonal elements $d^+_i$ of the) Laplacian: $O(|V|^2)$.
	\item The firing vector $f \in \mathbb{Z}^V$ can be determined in strongly polynomial time by Proposition \ref{prop::linequi.strongly.poly}. 
    \item  Having $f$, numbers $a_i$ and characteristic vectors of the sets $S_{a_i}$ can be determined in $O(|V|^2)$. (Note that $k \leq |V|$ as $\emptyset \subsetneq S_{a_1} \subsetneq S_{a_2} \subsetneq \dots \subsetneq S_{a_k} \subsetneq V$ is a strictly ascending chain.)
 	\item Any $x_j$ can be computed in $O(|V|^2)$. We need the vector $x_j$ for at most $2|V|$ instances of $j$.
   \item Checking whether $x_{a_{i + 1} - 1} \leadsto x_{a_{i + 1}}$ needs at most $|S_{a_j}| \leq |V|$ firings. Checking whether a vertex can fire can be done in $1$ step, since the out-degree of every vertex is already computed. Hence we can find in $|V|$ steps a vertex that can legally fire (if there is any). The effect of a firing can be computed in $O(|V|)$ steps. This means altogether $O(|V|^2)$ steps. 
\end{enumerate}
Note that the running time of the algorithm is dominated by the time required for computing the inverse of $L_v$, which is needed in Step (2).
\end{proof}

\section{General digraphs}

Unfortunately, the algorithm of Section \ref{s:Euler_multi} is not valid for non-Eulerian digraphs. It is conjectured by Bj\"orner and Lov\'asz in \cite{BL92} that the reachability problem is \NP-hard for general digraphs. In this section, we give two positive results: We show that the reachability problem is in \coNP, and we show a special case when it is decidable in polynomial time for general digraphs.

\begin{thm} Let $G$ be a digraph (with possibly multiple edges) and $x,y \in \Chip(G)$. Then deciding whether $x\leadsto y$ is in \coNP.
\end{thm}

\begin{proof}
As we noted in Section \ref{s:prelim}, the existence of $f\in \mathbb{Z}_+^V$ such that $y=x+Lf$ is a necessary condition for $x \leadsto y$, that can be checked in polynomial time. Hence in case there exists no $f\in \mathbb{Z}_+^V$ such that $y=x+Lf$, our certificate for $x \not\leadsto y$ is simply the statement that there exists no $f\in \mathbb{Z}_+^V$ such that $y=x+Lf$.

In case there exists $f\in \mathbb{Z}_+^V$ such that $y=x+Lf$, our certificate is a pair of vectors $f,g \in \mathbb{Z}^V$ satisfying the following properties. 
\begin{enumerate}
	\item $x + Lf = y$, $f \geq \mathbf{0}$ and $f\not\geq p$ for any non-zero period vector $p$ of $G$;
	\item $\mathbf{0} \leq g \leq f$, and there exists $v \in V$ such that $g(v) < f(v)$;
    \item For any $v \in V$, $g(v) = f(v)$ or
    $x_g(v) < d^+(v)$, where $x_g = x + L g$.
\end{enumerate}
All three conditions can be checked in polynomial time.

We claim that if $x \not\leadsto y$ and there exists $f\in \mathbb{Z}_+^V$ such that $y=x+Lf$ then such $f$ and $g$ exist.
Firstly, by subtracting an appropriate period vector of $G$ from the vector $f\in\mathbb{Z}_+^V$ with $y=x+Lf$, we can ensure that $f \geq \mathbf{0}$ and $f\not\geq p$ for any non-zero period vector $p$ of $G$.

Let $g$ be the firing vector of the maximal bounded chip-firing game from initial distribution $x$ with upper bound $f$.
By Lemma \ref{l:korlatos_jatek_moho}, $g$ is well defined.
By the definition of the bounded game, $\mathbf{0}\leq g\leq f$. If $g=f$ then $x\leadsto y$, hence if $x\not\leadsto y$, then necessarily there exists $v\in V$ such that $g(v)<f(v)$.
The third condition follows because $g$ is the firing vector of a maximal game with upper bound $f$.

Now we prove that if such an $f$ and $g$ exist then $x \not\leadsto y$. Suppose for a contradiction that $x \leadsto y$. By Lemma \ref{l:reach_with_reduced}, there exists a legal sequence $(v_1,v_2,\ldots,v_t)$ of firings with firing vector $f$ that leads from $x$ to $y$. Let $j$ be the largest index such that $\sum_{i = 1}^j \mathbf{1}_{v_i} \leq g$. Let $h$ be the firing vector of the sequence $(v_1,v_2,\ldots,v_j)$ and let $x_h = x + Lh$. By the choice of $j$, $g \geq h$ and $g(v_{j+1}) = h(v_{j+1}) < f(v_{j+1})$, hence 
\[
	x_g(v_{j+1}) - x_h(v_{j+1}) \geq L (g-h)(v_{j+1}) \geq 0.
\]
Since $(v_1,v_2,\ldots,v_j,v_{j+1})$ is a legal sequence of firings, we get 
\[
	d^+(v_{j+1}) \leq x_h(v_{j+1}) \leq x_g(v_{j+1}),
\] 
contradicting Condition 3.
\end{proof}

\subsection{Reachability of recurrent distributions}
\label{s:reach_recurr}

In this section, we show a case when the reachability problem can be decided in polynomial time also for general digraphs. More exactly, we give a case where the necessary condition of $x \sim y$ is also sufficient for $x \leadsto y$.
Our theorem uses the notion of recurrent chip-distributions.

\begin{defn}
We call a chip-distribution $x\in \Chip(G)$ \emph{recurrent} if there exists a non-empty sequence of legal firings that transforms $x$ to itself.
\end{defn}

\begin{thm} \label{thm::recurrent_reachability}
Let $G$ be a strongly connected digraph and $x,y \in \Chip(G)$. If $y$ is recurrent and $x \sim y$, then $x \leadsto y$.
\end{thm}
\begin{proof}
First we claim that if $x \sim y$ then there exists $f \in \mathbb{Z}_+^V$ such that $x=y+Lf$. Indeed, $x \sim y$ implies the existence of $g\in \mathbb{Z}^{V}$ with $x=y+Lg$. From Proposition \ref{prop::period}, for a sufficiently large $k \in \mathbb{Z_+}$, the vector $f = g + kp_G$ is nonnegative, while $x = y + Lg = y + Lg + k L p_G = y + Lf$ holds. 

Fix such an $f$. We proceed by induction on $\sum_{v\in V(G)}f(v)$. If $\sum_{v\in V(G)}f(v)=0$, then $x=y$, thus $x\leadsto y$.
Now suppose $\sum_{v\in V(G)}f(v)>0$. As $y$ is recurrent, there exists a sequence $(v_1,v_2,\dots,v_k)$ of legal firings from initial distribution $y$ (a vertex may occur multiple times), that leads back to $y$. Fix such a sequence.
We claim that in this sequence, each vertex occurs at least once. Indeed, for the firing vector $g$ of the game,  $y=y+Lg$ thus $g$ is a multiple of the primitive period vector, and the primitive period vector of a strongly connected digraph is strictly positive, as claimed in Proposition \ref{prop::period}. 

Let $i$ be the smallest index such that $f(v_i)>0$. Such an index exists because each vertex is listed at least once in $v_1,v_2,\dots,v_k$. Starting from $y$, fire the vertices $v_1, \dots , v_{i-1}$. This is a legal game by definition. Let the resulting distribution be $y'$. We claim that the sequence of firings $v_1, \dots , v_{i-1}$ is also legal starting from $x$. To prove this, it is enough to show that $x(v_j)\geq y(v_j)$ for all $1\leq j\leq i-1$. 
This is true, because $x(v_j) = y(v_j) + (Lf)(v_j)$, 
where $(Lf)(v_j) \ge 0$, since the only negative 
element in the row corresponding to $v_j$ is $L(v_j,v_j)$, but $f(v_j) = 0$. Hence the firing of the vertices $v_1, \dots , v_{i-1}$ from distribution $x$ is legal. Let the distribution obtained by this game be $x'$. Thus $x\leadsto x'$.

For $x'$ and $y'$, we also have $x'=y'+Lf$. At position $y'$, firing $v_i$ is legal, by definition of the sequence $v_1, \dots , v_k$. Denote by $y''$ the distribution we get by firing $v_i$ at $y'$. The distribution $y''$ is recurrent, since firing $v_{i+1} \dots, v_k, v_1, \dots, v_i$ is a legal game that leads back to $y''$.
Now for $x'$ and $y''$ we have $x'=y''+Lf'$, where  $f'=f-\mathbf{1}_{v_i}$.
This way $\sum_{v\in V(G)} f'(v)=\sum_{v\in V(G)} f(v)-1$, hence by the induction hypothesis, $x'\leadsto y''$. 

We claim that $y''\leadsto y$. Indeed, firing $v_{i+1}, \dots, v_k$ starting from $y''$ is a legal game that leads to $y$. We also have $x\leadsto x'$. 
By transitivity, we have $x\leadsto y$.
\end{proof}

This theorem raises the question of the complexity of deciding whether a given chip-distribution is recurrent. By results of Bj\"orner and Lov\'asz (Lemmas \ref{l:korlatos_jatek_moho} and \ref{l:per_vektor_kihagyhato}), a chip-distribution $x$ is recurrent if and only if there exists a nonzero primitive period vector $p$, such that started from $x$, the maximal chip-firing game with upper bound $p$ has firing vector $p$. For Eulerian digraphs, this can be checked in polynomial time (even if the digraph has multiple edges). However, for general digraphs, the complexity of deciding recurrence is open.

Our aim is now to generalize Theorem \ref{thm::recurrent_reachability} for weakly connected digraphs. 
Here, we need to use the stronger necessary condition of the existence of a nonnegative $f$ such that $y=x+Lf$. 
However, the condition of $y$ being recurrent is not enough to make this condition sufficient on general digraphs.
We show this by an example at the end of this section (Example \ref{e:recurr}).  
With a somewhat stronger condition, however, we can generalize Theorem \ref{thm::recurrent_reachability} to weakly connected digraphs.

\begin{thm} \label{thm::elerhetoseg_elegs_felt_alt_graf}
Let $G$ be a weakly connected digraph, and $x, y \in \Chip(G)$ be two 
chip-distri\-bu\-ti\-ons such that there exists $f \in \mathbb{Z}_+^V$ with $y 
= x + Lf$. Suppose that for each strongly connected component $G'= (V', E')$ of 
$G$, $f|_{V'} =\mathbf{0}$ or $y|_{V'} \in \Chip(G')$ is recurrent. 
  Then $x \leadsto y$.
\end{thm}

\begin{proof}
  Fix a nonnegative vector $f \in \mathbb{Z}_+^V$ with $y = x + Lf$.
  Let $V_1, V_2, \dots, V_k$ be a topological ordering of the strongly 
  connected components of $G$, i.e., $V=V_1 \cup \dots \cup V_k$, for each $i$ the digraph
  $G_i = (V_i, E|_{V_i\times V_i})$ is strongly connected, and there is 
  no directed edge from $v_i \in V_i$ to $v_j \in V_j$ if $i > j$. 
  
  Let $x'$ be the chip-distribution obtained from $x$ by passing 
  $f(u)\cdot \overrightarrow{d}(u, v)$ chips from $u$ to $v$ for 
  each pair of vertices $u, v \in V$ where $u$ and $v$ are in different 
  strongly connected components. Note that
  $x \not \sim x'$ is possible. The proof of the theorem is based on the following 
  lemma. 
  \begin{lemma}
    For each $i$, $x'|_{V_i} \sim y|_{V_i}$ on the 
    digraph $G_i$. 
    Moreover, if $y|_{V_i}$ is recurrent on $G_i$, then
    there exists a legal game on $G_i$ with firing 
    vector $f|_{V_i}$ that 
    transforms $x'|_{V_i}$  to $y|_{V_i}$.
  \end{lemma}
  \begin{proof}
    Let $L_i$ be the Laplacian matrix of $G_i$. We first prove 
    that $x'|_{V_i} \sim y|_{V_i}$ (as chip-distributions on $G_i$) by showing that 
    $x'|_{V_i} + L_i f|_{V_i} = y|_{V_i}$. 
    For this, let $v \in V_i$. 
    Then 
    \begin{equation*}
    \begin{split}
      x'(v) + (L_i f|_{V_i})(v) = \\ 
      x(v) + \sum_{v' \in V \setminus V_i} \left( 
        \overrightarrow{d}(v', v) \cdot f(v') - 
        \overrightarrow{d}(v, v') \cdot f(v) \right) 		+ (L_i f|_{V_i})(v) = \\ 
      x(v) + (Lf)(v) = y(v). 
    \end{split}
    \end{equation*}
    
    Now, if $y|_{V_i}$ is recurrent, by Theorem \ref{thm::recurrent_reachability},
    $x'|_{V_i}\leadsto y|_{V_i}$ in $G_i$.        
    Let $g_i\in \mathbb{Z}^{V_i}$ be the firing vector of a legal game transforming $x'|_{V_i}$ to $y|_{V_i}$. 
    Then $L_i (f|_{V_i} - g_i) = 0$, hence by Proposition \ref{prop::period}, $g_i - f|_{V_i} = c \cdot p_{G_i}$ with $c \in \mathbb{Z}$.
    If $c=0$ then $f|_{V_i}$ is the firing vector of a legal game, proving the lemma. 
    In the followings, we treat separately the cases  
    $c < 0$ and $c > 0$. 
    
    Suppose that $c < 0$. Since $y|_{V_i}$ is recurrent, there is a legal game on $G_i$ that transforms $y|_{V_i}$ back to itself. For the firing vector $g$ of this game, $L_i g=0$, hence $g=\lambda \cdot p_{G_i}$ with $\lambda\in \mathbb{Z}$, $\lambda>0$. 
By Lemma \ref{l:per_vektor_kihagyhato}, we can suppose that $\lambda=1$.
Now starting from distribution $x'|_{V_i}$ on $G_i$,    after playing the legal game with firing vector $g_i$, we get to the distribution $y|_{V_i}$.
Then iterate $-c$ times the legal game with firing vector $p_{G_i}$. This gives us a legal game with firing vector $f|_{V_i}$, finishing the proof for the $c < 0$ case.
    
Now suppose that $c>0$.
Then Lemma \ref{l:per_vektor_kihagyhato} guarantees that there is a legal game from $x'|_{V_i}$ with firing vector $g_i-c\cdot p_{G_i}=f|_{V_i}$.
This finishes the proof of the lemma.
\end{proof}

For each $1 \le i \le k$ let $f_i$ be the vector with $f_i(v) = f(v)$ if $v \in V_i$, and $f_i(v) = 0$ otherwise.
Let $s_i=\sum_{j\le i}f_j$, i.e., $s_i(v) = f(v)$ if $v \in \bigcup_{j\le i} V_j$, and $s_i(v) = 0$ otherwise. Let $x_i = x + L s_i$ and $x_0 = x$. 
We show that for $i =1, \dots, k$, starting from the distribution $x_{i - 1}$, there is a legal game on $G$ with firing vector $f_i$.
Since $x_{i-1} + L f_i=x_i$, and $x_k = y$, this is enough to finish the proof of the theorem. 

  So let $i$ be fixed. It is easy to see that for each $v\in V_i$ 
  \begin{equation}
  \label{e:1}
  x'(v) = x_{i - 1}(v) - 
  f(v) \cdot \sum_{v' \in V \setminus V_i} 
  \overrightarrow{d}(v, v').  
  \end{equation}
  
If $f|_{V_i} = \mathbf{0}_{V_i}$, then $f_i = \mathbf{0}$, hence we have nothing   to prove. If this is not the case, then $y|_{V_i}$ is recurrent by the assumptions. 
Using the lemma, from initial distribution $x'|_{V_i}$ there exists a legal game on $G_i$ with firing vector $f|_{V_i}$.
We claim that the same sequence of firings on $G$, with initial distribution $x_{i-1}$ remains a legal game. 
Indeed, we can see from \eqref{e:1} that by playing the game on $G$ from initial distribution $x_{i-1}$, at any moment we have a distribution that is greater or equal on $V_i$ than the distribution we get by playing the game on $G_i$ with initial distribution $x'|_{V_i}$.
Hence there exists a legal game on $G$ with initial distribution $x_{i-1}$ and firing vector $f_i$. This finishes the proof of the theorem. 
\end{proof}

Note that for distributions $y$ such that $y$ is recurrent restricted to each strongly connected component, Theorem \ref{thm::elerhetoseg_elegs_felt_alt_graf} gives a necessary and sufficient condition for the reachability of $y$. 
The condition of the theorem, i.e.~whether there exists $f \in \mathbb{Z}_+^V$ with $y = x + Lf$, can also be decided in polynomial time, as discussed in Section \ref{s:prelim}.

\begin{exmp} \label{e:recurr}
Now we give an example showing that Theorem \ref{thm::recurrent_reachability} does not remain true for general digraphs, i.e.~for general digraphs, the existence of $f\in \mathbb{Z}^V_+$ such that $y=x+Lf$ and 
$y$ being recurrent is not sufficient for $x\leadsto y$. 

Let $G$ be the following digraph:
$$V(G)=\{v_1,v_2,v_3,v_4,v_5,v_6\}$$
$$E(G)=\{\overrightarrow{v_1 v_2}, \overrightarrow{v_2 v_1},\overrightarrow{v_2 v_3}, \overrightarrow{v_3 v_2}, \overrightarrow{v_3 v_4}, \overrightarrow{v_4 v_3}, \overrightarrow{v_4 v_1}, \overrightarrow{v_1 v_4}, \overrightarrow{v_3 v_5}, \overrightarrow{v_4 v_6}, \overrightarrow{v_5 v_6}, \overrightarrow{v_6 v_5}\}$$

Let $x=(1,1,0,0,1,0)$ and $y=(0,0,1,1,1,0)$.
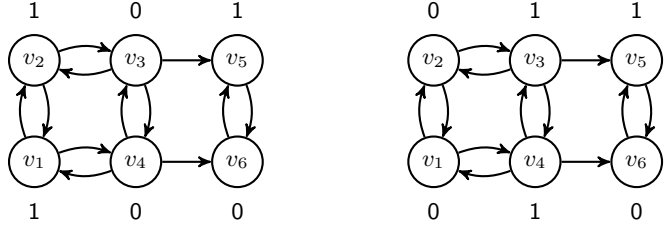
\begin{figure}[h]
\begin{center}
\begin{tikzpicture}[->,>=stealth',auto,scale=1.35,
                    thick,every node/.style={circle,draw,font=\sffamily\small}]
  \node[label=below:1] (1) at (-1, 0) {\small $v_1$};
  \node[label=above:1] (2) at (-1, 1) {\small $v_2$};
  \node[label=above:0] (3) at (0, 1) {\small $v_3$};
  \node[label=below:0] (4) at (0, 0) {\small $v_4$};
  \node[label=above:1] (5) at (1, 1) {\small $v_5$};
  \node[label=below:0] (6) at (1, 0) {\small $v_6$};
  \path[every node/.style={font=\sffamily\small}]
    (1) edge [bend left=20] node {} (2)
    (2) edge [bend left=20] node {} (1)
    (2) edge [bend left=20] node {} (3)
    (3) edge [bend left=20] node {} (2)
    (3) edge [bend left=20] node {} (4)
    (4) edge [bend left=20] node {} (3)
    (4) edge [bend left=20] node {} (1)
    (1) edge [bend left=20] node {} (4)
    (3) edge node {} (5)
    (4) edge node {} (6)
    (5) edge [bend left=20] node {} (6)
    (6) edge [bend left=20] node {} (5);
\end{tikzpicture}
\hspace{1.5cm}
\begin{tikzpicture}[->,>=stealth',auto,scale=1.35,
                    thick,every node/.style={circle,draw,font=\sffamily\small}]
  \node[label=below:0] (1) at (-1, 0) {\small $v_1$};
  \node[label=above:0] (2) at (-1, 1) {\small $v_2$};
  \node[label=above:1] (3) at (0, 1) {\small $v_3$};
  \node[label=below:1] (4) at (0, 0) {\small $v_4$};
  \node[label=above:1] (5) at (1, 1) {\small $v_5$};
  \node[label=below:0] (6) at (1, 0) {\small $v_6$};
  \path[every node/.style={font=\sffamily\small}]
    (1) edge [bend left=20] node {} (2)
    (2) edge [bend left=20] node {} (1)
    (2) edge [bend left=20] node {} (3)
    (3) edge [bend left=20] node {} (2)
    (3) edge [bend left=20] node {} (4)
    (4) edge [bend left=20] node {} (3)
    (4) edge [bend left=20] node {} (1)
    (1) edge [bend left=20] node {} (4)
    (3) edge node {} (5)
    (4) edge node {} (6)
    (5) edge [bend left=20] node {} (6)
    (6) edge [bend left=20] node {} (5);
\end{tikzpicture}
\label{fig::example}
\end{center}
\caption{The chip-distributions $x$ and $y$ on $G$}
\end{figure}

It is easy to see that $y$ is recurrent, since firing $v_5$ then firing $v_6$ transforms it back to itself.
Also, for the reduced $f=(1,1,0,0,0,0)$, $y=x+Lf$.
However, $x\not\leadsto y$, as for $x\leadsto y$ we need to be able to fire the firing vector $f$. However, neither $v_1$ nor $v_2$ can fire in $x$. 
\end{exmp}

\begin{remark} We note that our definition of recurrence is analogous to the so-called ``wrong definition of recurrence'' from \cite{Babai10}. 
Another possibility to define recurrence could be: $x$ is recurrent if for any chip-distribution $x'$ such that $x\leadsto x'$, we have $x'\leadsto x$.
However, we can see from Example \ref{e:recurr}, that on general digraphs, Theorem \ref{thm::recurrent_reachability} does not remain true for this notion, either. This suggests that the right notion of recurrence for general digraphs should be "recurrent restricted to each strongly connected component".
\end{remark}

\section{Open questions and related problems}
\label{s:open_prob}

The most intriguing open question in the area is the complexity of the reachability problem on general digraphs.
An interesting special case of this problem is deciding whether a chip-distribution on a general digraph is recurrent.

\begin{prob}
Let $G$ be a digraph and $x,y \in \Chip(G)$. What is the complexity of deciding whether $x \leadsto y$?
\end{prob}
\begin{prob}
Let $G$ be a digraph. What is the complexity of deciding whether a chip-distribution $x \in \Chip(G)$ is recurrent?
\end{prob}

We conjecture that both of these questions are \coNP-hard.

A related problem to the chip-firing reachability problem is the so-called chip-firing halting problem.

\begin{bekezdes}{Chip-firing halting problem} Given a digraph $G$ and a chip-distribution $x\in \Chip(G)$, decide if $x$ is terminating.
\end{bekezdes}

Informally, the halting problem and the reachability problem are both about determining the firing vector of a maximal game, only this game is a chip-firing game for the halting problem, and a bounded chip-firing game for the reachability problem.

The halting problem is known to be in \textbf{P} for simple Eulerian digraphs \cite{BL92}, and it is known to be \NP-complete for general digraphs \cite{farrell-levine-coeulerian}.
The complexity of the problem is open both for simple digraphs, and for Eulerian digraphs.
We point out the following:

\begin{prop}
The chip-firing halting problem is in \coNP for Eulerian digraphs.
\end{prop}
\begin{proof}
Our certificate for ``$x$ is non-terminating'' is a recurrent chip-distribution $y$ such that $x\sim y$. Since the graph is Eulerian, both the fact that $y$ is recurrent, and that $x \sim y$, can be checked in polynomial time. 
  
We show that $x$ is non-terminating if and only if such a $y$ exists. 
The number of chip-distributions reachable from $x$ by a legal game is finite (on each vertex, the number of chips needs to be between $0$ and $\sum_{v\in V}x(v)$). Hence if $x$ is nonterminating, starting a legal chip-firing game from $x$, we will eventually visit some chip-distribution $y$ twice. This $y$ is therefore recurrent. Moreover, $y\sim x$, since $y=x+Lf$ for the firing vector of the game leading from $x$ to $y$. 

For the other direction, we use a lemma of Bond and Levine \cite[Lemma 4.3.]{Bond-Levine}: if for two chip-distributions $x$ and $y$ on a strongly connected digraph $G$, $x\sim y$, then $x$ is terminating if and only if $y$ is terminating. Note that now $G$ is strongly connected since it is connected and Eulerian.
Note also that a recurrent chip-distribution is always non-terminating, since we can repeat the nonempty legal game transforming it back to itself indefinitely. Hence $y$ is non-terminating, consequently, $x$ is non-terminating.
\end{proof}

This means, that for Eulerian digraphs, the chip-firing halting problem is in $\NP\cap\coNP$.

\begin{prob}
Is there a polynomial time algorithm that decides the chip-firing halting problem for Eulerian digraphs (with multiple edges possible)?
\end{prob}

\section*{Acknowledgements}

We would like to thank Andr\'as Frank and \'Agoston Weisz for calling our attention to reachability questions and the definition of recurrent distributions. Investigations  of \'Agoston in some special cases \cite{weisz} were also very helpful for us.

\end{document}